\documentclass{siamltex}
\usepackage{graphicx}
\newcommand{\del}{\partial}
\renewcommand{\(}{\left(}
\renewcommand{\)}{\right)}
\newcommand{\arcsinh}{\mathop{\mathrm{arcsinh}}}
\newcommand{\R}{\mathbf{R}}
\renewcommand{\S}{\mathbf{S}}
\newcommand{\sgn}{\mathop{\mathrm{sgn}}}
\newcommand{\TLO}{T_{\scriptscriptstyle{LO}}}
\newcommand{\THI}{T_{\scriptscriptstyle{HI}}}
\newcommand{\TMAX}{T_{\scriptscriptstyle{MAX}}}
\font\bbb=msbm10
\renewcommand{\Re}{\hbox{\bbb\char82}}

\title{Computing Minimum Time Paths With Bounded Acceleration}

\author{Stewart D. Johnson\thanks{Department of Mathematics and Statistics, Williams College, Williamstown, MA, {\tt sjohnson@williams.edu}}}

\begin{document}
\maketitle

\begin{abstract}
  Solving for the minimum time bounded acceleration trajectory with prescribed position and velocity at endpoints is a highly nonlinear problem. The methods and bounds developed in this paper distinguish when there is a continuous acceleration solution and reduce the problem of computing the optimal trajectory to a search over two parameters, planar rotation $-\pi/2<\theta<\pi/2$ and spatiotemporal dilation $0<\alpha<\Lambda(\theta)$.
\end{abstract}

\begin{keywords}
minimal time paths, bounded acceleration, bilinear tangent law
\end{keywords}

\begin{AMS}
49M05, 49M37, 34K35, 34K28, 65K05
\end{AMS}

\section{Introduction}

We seek better understanding of and numerical methods for computing time-minimizing planar trajectories $(x(t),y(t))$ that have bounded acceleration $|(x'',y'')|\le 1$. A variety of boundary conditions can be considered. In this work we assume position and velocity are fully specified at initial and terminal points of the trajectory.

The problem of actually computing time minimizing trajectories has many difficulties. Continuous, bang-bang, and constant acceleration minimizers are all possible and can occur in close proximity to one another. Minimum time may depend discontinuously on boundary conditions, although the dependence is lower semi-continuous with constant acceleration solutions at points of discontinuity \cite{morgan}.

Optimizing using Pontryagin's principle yields a well-known stationarity condition (often called the bilinear tangent law). The stationarity condition is not sufficient, and multiple non-optimal solutions may exist. It is also possible to have local, not global, minimizers. In \cite{morgan} there is an example of boundary conditions with a continuum of stationary solutions containing a non-optimal local minimizer.

Even if it is known that the stationarity condition has a unique continuous acceleration solution, numerically approximating the solution is multidimensional and highly nonlinear \cite{feng}.

This work addresses these difficulties with the following contributions. First, we give necessary and sufficient conditions to determine when we have constant, bang-bang, or continuously varying acceleration solutions. Secondly, in the case of continuously varying acceleration, we reduce the numerical problem of computing the trajectory to a search of two continuous parameters over a semi-bounded region.

Minimizing time under constraints on acceleration is an interesting problem on its own \cite{bhat,aneesh,feng}, but also shows up in kinodynamic motion planning when acceleration is the dominant constraint, as in cases of limited traction \cite{marko}. The techniques in this work were applied to calculate the fastest path around the bases on a baseball diamond in \cite{morgan}, which garnered popular attention via National Public Radio, Huffington Post, Live Science, Science News, and Math Goes Pop.

\section{Time Minimizing Paths}

A time minimizing planar trajectory $(x(t),y(t))$ with bounded acceleration $|(x'',y'')|\le 1$, must satisfy \cite{morgan}: \begin{equation}
(x'',y'') = {At+B \over |At+B|}  \ \ \ \ A,B\in \Re^2 \label{eqn:morgan}
\end{equation}
on any open segment of the trajectory that is not restricted by boundary conditions.

This form subsumes the classic bilinear tangent law \cite{bryson,lewis}, and also contains bang-bang and constant acceleration solutions by setting $B=(0,0)$. Note that by rescaling space we can assume any positive bound on the magnitude of acceleration.

Assuming $B\not= (0,0)$ in (\ref{eqn:morgan}), then a rotation, spatiotemporal dilation, reflection, and time shift can transform acceleration to the form $(1,t)/\sqrt{1+t^2}$. Specifically, let
\begin{equation}
\begin{array}{lcl}
f''(t) = 1/\sqrt{1+t^2} &\;& g''(t) = t/\sqrt{1+t^2}\cr
f'(t)=\arcsinh(t) &\;& g'(t)= \sqrt{1+t^2} \cr
f(t)=t \arcsinh(t)-\sqrt{1+t^2} &\;& g(t)=  {1\over 2}\( t \sqrt{1+t^2}+\arcsinh(t)\) \cr
\end{array}
\label{eqn:fg}
\end{equation}
Then it follows:
\begin{proposition}
If $(\tilde x(t),\tilde y(t))$ is a minimal time path with unit magnitude continuous acceleration for $-\epsilon<t<\epsilon$ then there exist unique values for $\alpha>0$, $\theta\in(-\pi/2,\pi/2)$, $\sigma,\eta=\pm 1$, and $t_0$ such that
$$ \(\matrix{  x(t) \cr  y(t)}\) = {1\over\alpha^2}\R_\theta \S
\(\matrix{ f(\alpha (t-t_0))+u_0 \alpha (t-t_0) \cr  g(\alpha (t-t_0))+v_0 \alpha (t-t_0) }\)$$
for $-\epsilon<t<\epsilon$, with rotations and reflections
$$\R_\theta = \left[\matrix{ \cos(\theta) & -\sin(\theta) \cr \sin(\theta) & \cos(\theta)}\right] \qquad
\S = \left[\matrix{ \sigma & 0 \cr 0 & \eta }\right]
$$
\label{prop:form}
\end{proposition}

This formulation of the solution preserves time direction, and uniquely covers all possibilities by a $180^\circ$ sweep of space together with vertical and horizontal flips.

A variety of boundary conditions can be considered. This work focuses on taking position and velocity specified at initial and terminal locations:
\begin{equation}
\begin{array}{rcl}
\(\matrix{  u_1 \\  v_1 }\)&=&\(\matrix{  x'(0) \\  y'(0)}\)  \\[16pt]
\(\matrix{  u_2 \\  v_2 }\)&=&\(\matrix{  x'(T) \\  y'(T)}\)   \\[16pt]
\(\matrix{ \delta_{ x} \\ \delta_{ y} }\)&=& \(\matrix{  x(T)- x(0) \\  y(T)- y(0)}\),
\end{array}
\label{eqn:bndry}
\end{equation}
where $T$ is free and minimized subject to $|(x'',y'')|\le 1$

There is a unique time-minimizing trajectory: existence is established by bounded Lipschitz convergence, and uniqueness follows in that averaging the acceleration of two solutions must also be a solution with maximal acceleration (see \cite{morgan} for details).

Bang-bang or constant solutions happen when $B=(0,0)$ in (\ref{eqn:morgan}) and don't fit the formulation in proposition \ref{prop:form}. A complete characterization of boundary values when bang-bang and constant solutions exist is given in the following due to Frank Morgan.

Fitting (\ref{eqn:morgan}) to any bang-bang or constant acceleration solution has $B=(0,0)$ and $A$ is parallel to the difference in endpoint velocities. The problem is computationally simplified by rotating the plane so the difference in endpoint velocities is horizontal.
\begin{proposition}[F. Morgan]
Assuming $v_1=v_2=v$ in (\ref{eqn:bndry}), the minimum time constant acceleration problem has bang-bang or constant acceleration solutions in precisely the following three cases:
\begin{enumerate}
\item  $u_2=u_1$ and $(\delta_x, \delta_y)$ is nonzero,
\item  $\delta_y = v =0$ and $\delta_x \not =  |u_2^2-u_1^2|/2$,
\item  $v\not = 0$, $u_1\not=u_2$, and $\delta_x =\delta_y (u_1+u_2)/2v \,\pm\, ((\delta_y/v)^2 - (u_2-u_1)^2)/4$.
\end{enumerate}
\label{prop:morgan}
\end{proposition}

\begin{proof}
If the initial and final velocities are equal, $u_2=u_1$, and $\delta_x=\delta_y=0$ then total time is zero. If $(\delta_x,\delta_y)$ is non-zero, it is straightforward to construct a bang-bang solution with acceleration reversing direction at the halfway point $(\delta_x/2,\delta_y/2)$.  If $v=0$ it is again a straightforward exercise to construct a bang-bang or constant acceleration solution.

Henceforth we assume initial and final velocities are different and $v\not=0$.

Assume that the solution is bang-bang, with $(x'',y'')=(1,0)$ for time $T_1\ge0$ and
then $(x'',y'')=(-1,0)$ for time $T_2\ge 0$. Then compute
$$
\begin{array}{rcl}
\delta_x &=& T_1(u_1+T_1/2) + T_2(u_2+T_2/2), \\
\delta_y/v &=& T_1 + T_2, \\
u_2-u_1 &=& T_1 - T_2. \\
\end{array}
$$
Solving this system of equations and allowing for the reversed order of acceleration yields (3). Note that $((y/v)^2 - (u_2-u_1)^2)=0$ iff  $T_1$ or $T_2$ is zero, yielding a constant acceleration solution.

Conversely, if a solution has $\delta_x =\delta_y (u_1+u_2)/2v \pm ((\delta_y/v)^2 - (u_2-u_1)^2)/4$, then there is a bang-bang or constant solution with acceleration of the form $(\pm 1,0)$. This is the unique minimizer for the horizontal dimension of the problem: $x'(0)=u_1$, $x'(T)=u_2$, $x(T)-x(0)=\delta_x$. Allowing any vertical component to acceleration would reduce the magnitude of horizontal acceleration, and so would take more time.

\end{proof}

The bang-bang and constant acceleration solutions are thus completely characterized and straightforward to calculate. However, computing solutions in the continuous acceleration case is a highly nonlinear problem. Using the above formulation, the problem can be reduced to a search over two continuous parameters (rotation and dilation) and one discrete (vertical flip). This is a significant improvement over other proposed methods \cite{feng,aneesh,marko}. The method is outlined here, and detailed in the remainder of the paper.

Continuing with the assumption $v_1=v_2=v$, we can reflect about the $y$-axis to assume $u_2-u_1>0$ and rescale space and time (see section \ref{sec:norm}) so that $u_2-u_1=1$. This produces the normalized boundary values:
\begin{equation}
\begin{array}{ccc}
\(\matrix{  u \\  v }\)&=&\(\matrix{  x'(t_1) \\  y'(t_1)}\)  \\[16pt]
\(\matrix{  u+1 \\  v }\)&=&\(\matrix{  x'(t_2) \\  y'(t_2)}\)  \\[16pt]
\(\matrix{ \delta_{ x} \\ \delta_{ y} }\)&=& \(\matrix{  x(t_2)- x(t_1) \\  y(t_2)- y(t_1)}\),
\end{array}
\label{eqn:normbndry}
\end{equation}
where $t_1>t_2$ are free and $t_2-t_1$ is minimized under unit magnitude acceleration. Proposition \ref{prop:form} implies the existence of a solution of the form
\begin{equation}
\(\matrix{  x(t) \cr  y(t)}\) = {1\over\alpha^2}\R_\theta \S
\(\matrix{ f(\alpha t)+u_0 \alpha t \cr  g(\alpha t)+v_0 \alpha t }\)
\label{eq:difvel}
\end{equation}
With $x'(t_2)-x'(t_1)=1$ and $y'(t_2)-y'(t_1)=0$ we get
\begin{equation}
\(\matrix{  \alpha\cos(\theta) \cr  \alpha\sin(\theta)} \) = \left[\matrix{\sigma & 0 \cr 0 & \eta}\right]
\(\matrix{ f'(\alpha t_2)-f'(\alpha t_2) \cr  g'(\alpha t_2)-g'(\alpha t_1) }\)
\label{eqn:difval}
\end{equation}
For $-\pi/2<\theta<\pi/2$ and $f'$ monotone increasing we must have horizontal orientation $\sigma=+1$. Given any $\alpha>0$, $\theta\in(-\pi/2,\pi/2)$, and vertical orientation $\eta=\pm1$, equations (\ref{eqn:difval}) can be rapidly solved to any precision. This is shown in section \ref{sec:time} where we first solve for $\tau_1=\alpha t_1$ as a root of a monotone function with initial upper bounds $\TLO<\tau_1<\THI$, after which we get $t_2$, $u_0$, and $v_0$ by direct computation.

Solving (\ref{eqn:normbndry}) is thus reduced to a search over $\theta\in(-\pi/2,\pi/2)$, $\alpha>0$, and $\eta=\pm1$ to match the displacements $\delta_x$, $\delta_y$.

We can also get an upper bound for $\alpha$.
\begin{definition} For $T>1$ and $-\pi/2<\theta<\pi/2$, let
\begin{equation}
\Lambda(T,\theta)= \max\{\alpha>0\; \big|\; \alpha T  > \exp(\alpha\cos\theta/2)-1
\label{eqn:alphamaxUP}
\end{equation}
\end{definition}

Values for $\Lambda$ can be rapidly calculated. The following is established in section \ref{sec:time}
\begin{proposition}
If $t_1,t_2,u_0,v_0,\alpha,\theta$ solve the boundary conditions (\ref{eqn:normbndry}), then $\alpha<\Lambda(t_2-t_1,\theta)$
\label{prop:lamb}
\end{proposition}

An upper bound for $\alpha$ results from using an a priori upper bound $\TMAX$ to total time $t_1-t_2$. Such a bound can be constructed from a zigzag trajectory with two zero velocity points (see section \ref{sec:dilate}).

Normalization is carefully defined in the next section, and the above bounds and propositions are developed in section \ref{sec:num}.

\section{Normalization}\label{sec:norm}
\subsection{Transformations}
Suppose $(\tilde x (\tilde t), \tilde y (\tilde t))$ is a minimal time curve with
\begin{equation}
\begin{array}{rcl}
\(\matrix{ \tilde x'(\tilde t_1) \\ \tilde y'(\tilde t_1)}\) = \(\matrix{ \tilde u_1 \\ \tilde v_1 }\) \\[16pt]
\(\matrix{ \tilde x'(\tilde t_2) \\ \tilde y'(\tilde t_2)}\) = \(\matrix{ \tilde u_2 \\ \tilde v_2 }\) \\[16pt]
\(\matrix{ \tilde x(\tilde t_2)-\tilde x(\tilde t_1) \\ \tilde y(\tilde t_2)-\tilde y(\tilde t_1)}\) = \(\matrix{ \delta_{\tilde x} \\ \delta_{\tilde y} }\) \end{array}
\end{equation}
and acceleration
$$\sqrt{\tilde x''(\tilde t)^2+\tilde y''(\tilde t)^2}=1$$

Then given any rotation angle $\phi$, dilation $\beta>0$, reflections $\tilde\sigma,\tilde\eta=\pm 1$, and transforming
\begin{equation}
\(\matrix{ x(t) \cr y(t) }\) = {1\over\beta^2} \R_\phi \tilde\S \(\matrix{ \tilde x(\beta t) \cr \tilde y(\beta t) }\)
\end{equation}
yields a minimal time path $(x(t),y(t))$ that satisfies boundary conditions
\begin{equation}
\begin{array}{rcl}
t_1 &=& \tilde t_1/\beta\\
t_2 &=& \tilde t_2/\beta\\[10pt]
\(\matrix{ x'(t_1) \\  y'(t_1)}\) &=& {1\over\beta}\R_\phi \tilde\S\(\matrix{ \tilde u_1 \\ \tilde v_1 }\) \\[16pt]
\(\matrix{ x'(t_2) \\  y'(t_2)}\) &=& {1\over\beta}\R_\phi \tilde\S\(\matrix{ \tilde u_2 \\ \tilde v_2 }\) \\[16pt]
\(\matrix{ x(t_2)- x(t_1) \\ y(t_2)- y(t_1)}\) &=& {1\over\beta^2}  \R_\phi\tilde\S\(\matrix{ \delta_{\tilde x} \\ \delta_{\tilde y} }\) \end{array}
\end{equation}
and acceleration
$$\sqrt{ x''(t)^2+ y''(t)^2}=1$$

\subsection{The Normalized Problem}
Given real-world boundary values $\tilde u_1$, $\tilde v_1$, $\tilde u_2$, $\tilde v_2$, $\delta_{\tilde x}$, $\delta_{\tilde y}$, let
$$
\begin{array}{rcl}
\beta  &=& \((\tilde u_2-\tilde u_1)^2+(\tilde v_2-\tilde v_1)^2\)^{1/2}\\
\phi   &=& \arctan(\sigma {\tilde v_2- \tilde v_1 \over \tilde u_2 - \tilde u_1})\in(-{\pi\over 2},{\pi\over 2})\\
\sigma &=& \sgn(\tilde u_2-\tilde u_1)\\
\eta   &=& 1
\end{array}
$$
Applying the linear transformation ${1\over\beta^2} \R_\phi \tilde\S$ to the $\tilde x, \tilde y$ system and scaling time $\tilde t=\beta t$ yields boundary values
$$
\begin{array}{rclcl}
{1\over\beta} \R_\phi \tilde\S \(\matrix{ \tilde u_1 \\ \tilde v_1 }\) &=&   \(\matrix{ u_1 \\ v_1 }\) \\[16pt]
{1\over\beta} \R_\phi \tilde\S \(\matrix{ \tilde u_2 \\ \tilde v_2 }\) &=&   \(\matrix{ u_2 \\ v_2 }\)
&=& \(\matrix{ u_1+1 \\ v_1 }\) \\[16pt]
{1\over\beta^2} \R_\phi \tilde\S \(\matrix{ \delta_{\tilde x} \\ \delta_{\tilde y} }\) &=&   \(\matrix{ \delta_{x} \\ \delta_{y} }\) \\[10pt]
\end{array}
$$
with unit acceleration.

Note that a problem with $(u_1,v_1)$ equal to $(u_2,v_2)$ cannot be normalized. In this case the solution is bang-bang (proposition \ref{prop:morgan}).

The normalized problem is thus to estimate the minimal time path $(x(t),y(t))$ with unit acceleration and boundary conditions determined by $u_1,v_1,\delta_x,\delta_y$:
\begin{equation}
\begin{array}{rcl}
\(\matrix{  x'(t_1) \\  y'(t_1)}\) &=& \(\matrix{  u_1 \\  v_1 }\) \\[16pt]
\(\matrix{  x'(t_2) \\  y'(t_2)}\) &=&  \(\matrix{  u_1+1 \\  v_1 }\)  \\[16pt]
\(\matrix{  x(t_2)- x(t_1) \\  y(t_2)- y(t_1)}\) &=& \(\matrix{ \delta_{ x} \\ \delta_{ y} }\) \end{array}
\end{equation}

Given a solution $(x(t),y(t))$ for $t_1<t<t_2$ to the normalized problem, we transform back to original coordinates $(\tilde x (\tilde t),\tilde y(\tilde t))$ as
\begin{equation}
\(\matrix{ \tilde x(\tilde t) \cr  \tilde y(\tilde t)}\) = \beta^2\tilde\S\R_\phi^{-1}\(\matrix{  x(\tilde t/\beta) \cr  y(\tilde t/\beta }\)
\end{equation}
for
$$\beta t_1=\tilde t_1<\tilde t<\tilde t_2=\beta t_2$$

\section{Numerics}\label{sec:num}
\subsection{Solving the Normalized Problem}
To solve the normalized problem, numerical methods are developed to calculate values for six parameters
$$
\begin{array}{ll}
\hbox{Normalized Time Interval:} & [\tau_1,\tau_2]\\
\hbox{Velocity Translation:} & (u_0,v_0)\\
\hbox{Time/Space Dilation:} & \alpha>0 \\
\hbox{Vertical Reflection:}  & \eta = \pm 1\\
\hbox{Rotation Angle:}     & -{\pi\over 2}<\theta<{\pi\over 2}.
\end{array}
$$
to satisfy six constraint equations
\begin{eqnarray}
{\alpha}\R_\theta^{-1}\( \matrix{ u_1 \cr v_1 }\) &=&
\(\matrix{  f'(\tau_1)+u_0 \cr \eta g'(\tau_1)+ v_0}\)\label{eqn:s1}\\[10pt]
{\alpha}\R_\theta^{-1}\( \matrix{ u_2 \cr v_2 }\) &=&
\(\matrix{  f'(\tau_2)+u_0 \cr \eta g'(\tau_2)+v_0}\)\label{eqn:s2}\\[10pt]
{\alpha^2}\R_\theta^{-1}\( \matrix{ \delta_x \cr \delta_y }\) &=&
\(\matrix{  f(\tau_2)-f(\tau_1) +u_0 (\tau_2-\tau_1) \cr \eta g(\tau_2)-\eta g(\tau_1) +v_0 (\tau_2 - \tau_1) }\)\label{eqn:s3}
\end{eqnarray}
for given boundary conditions $u_1,v_1,u_2,v_2,\delta_x,\delta_y$, with
\begin{equation}
\begin{array}{rcl}
u_2-u_1&=&1\\
v_2-v_1&=&0
\end{array}
\label{eqn:uvnorm}
\end{equation}

Subtracting equation (\ref{eqn:s1}) from (\ref{eqn:s2}) and using (\ref{eqn:uvnorm}) yields
\begin{equation}
\(\matrix{f'(\tau_2)-f'(\tau_1)\cr g'(\tau_2)-g'(\tau_1)}\) = \alpha \R_\theta^{-1} \(\matrix{ 1\cr 0}\) = \(\matrix{ \alpha\,\cos\theta \cr -\alpha\,\sin\theta}\)\label{eqn:deluv}
\end{equation}
which defines a map $(\alpha,\theta)\mapsto(\tau_1,\tau_2)$ that is independent of all boundary conditions. Given $\alpha,\theta$ this map can be quickly solved to arbitrary precision as shown in section \ref{sec:time}. For multiple calculations of the same precision, an interpolated hash table may be used.

Using boundary velocity conditions, integration constants $u_0, v_0$ follow from (\ref{eqn:s1}) (or (\ref{eqn:s2})), and we thus get a map $(\theta,\alpha)\mapsto (\mu_x,\mu_y)$ as
$$ \( \matrix{ \mu_x \cr \mu_y }\) = \(\matrix{  f(\tau_2)-f(\tau_1) +u_0 (\tau_2-\tau_1) \cr g(\tau_2)-g(\tau_1) +v_0 (\tau_2 - \tau_1) }\)$$
Then the minimal time solution is specified by finding the correct $(\theta, \alpha)$ to match the computed displacement $(\mu_x, \mu_y)$ to the target displacement $(\delta_x,\delta_y)$.

\subsection{Bounds on Dilation}\label{sec:dilate}
Fix $\theta\in(-\pi/2,\pi/2)$, and let
$$\( \matrix{ \rho_u \cr \rho_v }\)=\R_\theta^{-1}\( \matrix{ 1 \cr 0 }\) = \( \matrix{ \cos \theta \cr -\sin \theta }\)$$
so that $\rho_u>0$

Then (\ref{eqn:deluv}) yields
\begin{eqnarray}
\alpha\rho_u &=& \arcsinh(\tau_2)-\arcsinh(\tau_1)\label{eqn:asnh}\\[10pt]
\alpha\rho_v &=& \sqrt{1+\tau_2^{\,2}} - \sqrt{1+\tau_1^{\,2}}\label{eqn:vee}
\end{eqnarray}
Note that (\ref{eqn:asnh}) and $\alpha\rho_u>0$ make $\tau_2-\tau_1>0$. Recall that
\begin{equation}
\arcsinh(z)=\ln\(z+\sqrt{1+z^{\,2}}\)\label{eqn:lnsnh}
\end{equation}
Three readily verifiable bounds will be useful:
\begin{equation}
|z|<\sqrt{1+z^2}<|z|+1 \label{eqn:veebnd}
\end{equation}
\begin{equation}
z+|z| < z+\sqrt{1+z^2} \label{eqn:hbnd}
\end{equation}
and if $z>0$ then
\begin{equation}
z+\sqrt{1+z^2}<1+2z \label{eqn:zposhbnd}
\end{equation}

The next lemma follows from (\ref{eqn:vee}) and (\ref{eqn:veebnd}).
\begin{lemma}
$$\alpha\rho_v - 1 < |\tau_2| - |\tau_1| < \alpha\rho_v + 1$$
\label{clm:diffbound}
\end{lemma}

Proposition \ref{prop:lamb} is a corollary of the following lemma:
\begin{lemma}
$$\tau_2-\tau_1 > \exp(\alpha\rho_u/2)-1$$
\label{clm:sumbound}
\end{lemma}

\begin{proof}
From (\ref{eqn:asnh}),
$$
\alpha\rho_u = \arcsinh(\tau_2)-\arcsinh(\tau_1)
$$
An exercise of calculus demonstrates that for $\delta>0$ the maximum of $\arcsinh(\tau+\delta)-\arcsinh(\tau-\delta)$ is realized at $\tau=0$, hence
\begin{eqnarray*}
\alpha\rho_u &=& \arcsinh(\tau_2)-\arcsinh(\tau_1)\\
&\le& \arcsinh\((\tau_2-\tau_1)/ 2\)-\arcsinh\((\tau_2-\tau_1)/ 2\)\\
&=& 2\arcsinh\((\tau_2-\tau_1)/ 2\)\\
&\le& 2\ln(1+\tau_2-\tau_1)
\end{eqnarray*}
with the last step following from (\ref{eqn:lnsnh}) and (\ref{eqn:zposhbnd}).

\end{proof}

An a priori upper bound for $\tau_2-\tau_1$ comprised of three straight line segments joined at points of zero velocity. It takes a minimum of $|(u_1,v_1)|$ time units to bring initial velocity down to zero, and a minimum of $|(u_2,v_2)|$ to build up to final velocity from zero velocity. Connecting the two points of zero velocity with a straight bang-bang trajectory produces:
$$
\begin{array}{c}
\mu_1 = |(u_1,v_1)| \qquad \mu_2 = |(u_2,v_2)| \\
\TMAX = \mu_1 + \mu_2 +\sqrt{2}\((2\delta_x-\mu_1 u_1-\mu_2 u_2)^2+(2\delta_y-\mu_1 v_1-\mu_2 v_2)^2\)^{1/4}
\end{array}
$$
Thus the desired $\theta,\alpha$ solution will satisfy $\alpha<\Lambda(\TMAX,\theta)$.

\subsection{Solving for Time} \label{sec:time}

Fix $\theta\in(-\pi/2,\pi/2),\alpha>0$, and let
$$\( \matrix{ \mu_u \cr \mu_v }\)=\alpha \R_\theta^{-1}\( \matrix{ 1 \cr 0 }\)= \(\matrix{ \alpha\,\cos\theta \cr -\alpha\,\sin\theta}\)
$$
Then constraint (\ref{eqn:deluv}) becomes:
\begin{equation}
\( \matrix{ \mu_u \cr \mu_v }\)=\(\matrix{\arcsinh(\tau_2)-\arcsinh(\tau_1)\cr \sqrt{1+\tau_2^{\,2}}-\sqrt{1+\tau_1^{\,2}}}\)\label{eqn:mu}
\end{equation}
Note that $\mu_v>0$ implies $\tau_2>\tau_1$.

For simplicity, take $G(\tau)=g'(\tau)=\sqrt{1+\tau^2}$. With
\begin{equation}
\tau_2=\sinh(\arcsinh(\tau_1)+\mu_u),
\label{eqn:tau2}
\end{equation}
equation (\ref{eqn:mu}) reduces to
\begin{equation}
\mu_v= G(\sinh(\arcsinh(\tau_1)+\mu_u))-G(\tau_1)
\label{eqn:tau1}
\end{equation}

\begin{lemma} Fix $\mu_u>0$, then
\begin{equation}
\mu_v=G(\sinh(\arcsinh(\tau)+\mu_u))-G(\tau)\label{eqn:mono}
\end{equation}
is monotone in $\tau$, with $\mu_v\to\infty$ as $\tau\to\infty$ and $\mu_v\to-\infty$ as $\tau\to-\infty$.
\end{lemma}

\begin{proof}
Computing:
\begin{eqnarray*}
\lefteqn{{\del^2\over \del \delta^2} G(\sinh(\arcsinh(\tau)+\delta))}\\
&=&{4\cosh(2\delta+2\arcsinh \tau)+\cosh(4\delta+4\arcsinh \tau)+3\over (2\cosh(2\delta+2\arcsinh \tau)+2)^{3/2} }\\
&>&0
\end{eqnarray*}
hence for $\mu_u>0$,
$$ \left.{\del\over\del\delta} G(\sinh(\arcsinh(\tau)+\delta))\right|_{\delta=0}^{\mu_u}>0$$
making
$$
\begin{array}{c}
G'(\sinh(\arcsinh(\tau)+\delta))\sinh'(\arcsinh(\tau)+\delta)\qquad\qquad
\\ \qquad\qquad -G'(\sinh(\arcsinh(\tau)))\sinh'(\arcsinh(\tau))>0
\end{array}
$$
and
$$ G'(\sinh(\arcsinh(\tau)+\delta)){\sinh'(\arcsinh(\tau)+\delta)\over \sinh'(\arcsinh(\tau))} -G'(\tau)>0$$
thus
$${\del\mu_u\over\del \tau} = G'(\sinh(\arcsinh(\tau)+\delta))\sinh'(\arcsinh(\tau)+\delta){\arcsinh}'(\tau)-G'(\tau)>0$$

\end{proof}

As a corollary, we have:
\begin{lemma}
For any given $(\mu_u,\mu_v)$ with $\mu_u > 0$ there is a unique solution $(\tau_1,\tau_2)$ to (\ref{eqn:mu}).
\end{lemma}

The unique solution to (\ref{eqn:mu}) is estimated using bisection with the following bounds to initiate the algorithm.

\begin{lemma} Given $\mu_u>0$, $\mu_v$, define
\begin{equation}
\begin{array}{lcl}
\TLO &=& -e^{\mu_u}\max\{{1\over 2},{1-\mu_v\over\mu_u}\}\\
\THI &=& \max\{0,{1+\mu_v\over\mu_u}\},
\end{array}
\label{eqn:tlothi}
\end{equation}
then the unique solution $\tau$ to equation (\ref{eqn:mono}) satisfies $\TLO<\tau<\THI$
\label{clm:tlothi}
\end{lemma}

The proof consists of analyzing the three cases $\tau_1<\tau_2<0$, $\tau_1<0<\tau_2$, and $0<\tau_1<\tau_2$, as contained in the following three lemmas. Given $\mu_u>0$, $\mu_v$, let $\tau_1$ be the solution to equation (\ref{eqn:mono}) and $\tau_2=\sinh(\arcsinh(\tau_1)+\mu_u)$.

\begin{lemma}
If $\tau_1<0<\tau_2$ then $-{1\over 2} e^{\mu_u}<\tau_1$
\end{lemma}
\begin{proof}
$$
\begin{array}{rcl}
\mu_u &=& \arcsinh(\tau_2)-\arcsinh(\tau_1)\\
&>& -\arcsinh(\tau_1)
\end{array}
$$
Hence using $-{1\over 2} e^z < \sinh(z)$,
$$-{1\over 2} e^{\mu_u}<\sinh(-\mu_u)<\tau_1 $$

\end{proof}

\begin{lemma}
If $\tau_1<\tau_2<0$ then $\mu_v<0$ and $\tau_1>-e^{\mu_u}{1-\mu_v\over \mu_u}$
\label{clm:belowzero}
\end{lemma}
\begin{proof} With $\tau_1 = \sinh(\arcsinh(\tau_1))$ and $\tau_2 = \sinh(\arcsinh(\tau_1)+\mu_u)$,
$$\tau_2-\tau_1>\mu_u \left.{d\over dz} \sinh(z)\right|_{\arcsinh(\tau_2)} = \mu_u \cosh(\arcsinh(\tau_1)+\mu_u) $$
Lemma \ref{clm:diffbound} has $\mu_v-1<|\tau_2|-|\tau_1|$, and with $\tau_1<\tau_2<0$,
$$-\mu_v+1>-|\tau_2|+|\tau_1|=\tau_2-\tau_1$$
hence using ${1\over 2}e^{-z}<\cosh(z)$,
$$-\mu_v+1 > \mu_u \cosh\(\arcsinh(\tau_1)+\mu_u\) > {\textstyle{1\over 2}}\mu\exp(-\arcsinh(\tau_1)-\mu_u)$$
and
$$\arcsinh(\tau_1)>-\ln\(2e^{\mu_u}\(1-\mu_v\over\mu_u\)\).$$
Using $\sinh(u)>-{1\over2}e^{-u}$ completes the proof.

\end{proof}

\begin{lemma}
If $0<\tau_1<\tau_2$ then $\tau_1<{\mu_v+1\over\mu_u}$
\end{lemma}
\begin{proof}
Similar to the proof of claim \ref{clm:belowzero},
$$\tau_2-\tau_1>\mu_u \left.{d\over dz} \sinh(z)\right|_{\arcsinh(\tau_1)} = \mu_u \cosh(\arcsinh(\tau_1))>\mu_u \tau_1 $$
With $0<\tau_1<\tau_2$, claim \ref{clm:diffbound} implies $\mu_v+1>\tau_2-\tau_1$, and the result follows.

\end{proof}

\section{Conclusion}

Bounds and methods for solving the minimum time bounded acceleration path in the plane subject to velocity and location endpoint conditions are presented in this paper. An example implementation in Python is available from the author.

The methods will apply to other boundary restrictions, such as zero initial velocity, or free endpoint location, and the author would appreciate being informed of any adaptations.

\end{document}